\documentclass[11pt,oneside]{article}

\usepackage{amsthm,amsfonts, amsbsy, amssymb,amsmath,graphicx}
\usepackage{graphics}
\usepackage[english]{babel}
\usepackage{graphicx}
\usepackage{lineno}
\usepackage{enumerate}
\usepackage{tikz}

\newtheorem{theorem}{Theorem}
\newtheorem{lemma}[theorem]{Lemma}

\newtheorem{them}{Theorem}
\newtheorem{lema}[them]{Lemma}

\theoremstyle{definition}

\theoremstyle{remark}
\newtheorem{remark}[theorem]{Remark}

\title{On the Roman Bondage Number of Graphs on surfaces}

\author {Vladimir Samodivkin\vspace{5mm}\\
Department of Mathematics, University of Architecture Civil
Engineering and Geodesy\\
Hristo Smirnenski 1 Blv., 1046 Sofia, Bulgaria\\
{\tt vlsam\_fte@uacg.bg}\vspace{3mm}\\}

\begin{document}

\maketitle

\begin{abstract}
  A Roman dominating function on a graph $G$ 
is a labeling $f : V(G) \rightarrow \{0, 1, 2\}$ such
that every vertex with label $0$ has a neighbor with label $2$. 
  The Roman domination number, $\gamma_R(G)$, of
$G$ is the minimum of $\Sigma_{v\in V (G)} f(v)$ over such functions.
The Roman bondage number $b_R(G)$ is  the cardinality of 
a smallest set of edges whose removal  from $G$ 
results in a graph with Roman domination number not equal to  $\gamma_R(G)$.
 In this paper we obtain upper bounds on $b_{R}(G)$ in terms of 
 (a) the average degree and maximum degree, 
  and
 (b) Euler characteristic, girth and maximum degree. 
  	We also show that the Roman bondage number of every  graph which admits 
  a $2$-cell embedding on a surface with non negative Euler characteristic does not exceed $15$. 
\end{abstract}

\smallskip



\section{Introduction} \label{Intro}

All graphs considered in this article are finite, undirected, without loops or multiple edges. 
We denote the vertex set and the edge set of a graph $G$ by $V(G)$ and $ E(G),$  respectively. 
	Let $P_n$	 denote the path with $n$ vertices. 
	For any vertex $x$ of a graph $G$,  $N_G(x)$ denotes the set of all  neighbors of $x$ in $G$,  
	$N_G[x] = N_G(x) \cup \{x\}$ and the degree of $x$ is $d_G(x) = |N_G(x)|$. 
	The minimum and maximum degrees of the graph $G$ are denoted by $\delta(G)$ and $\Delta(G)$, respectively.
   For a graph $G$, let $x \in X \subseteq V(G)$. 
   A vertex $y \in V(G)$ is a $X$-private neighbor  of $x$ if $N_G[y] \cap X = \{x\}$. 
		The  $X$-private neighborhood  of $x$, denoted $pn_G[x,X]$, is  the set of all $X$-private neighbors of $x$.
An orientable compact 2-manifold $\mathbb{S}_h$ or orientable surface $\mathbb{S}_h$ (see \cite{Ringel}) of genus $h$ is
obtained from the sphere by adding $h$ handles. Correspondingly, a non-orientable compact
2-manifold $\mathbb{N}_q$ or non-orientable surface $\mathbb{N}_q$ of genus $q$ is obtained from the sphere by
adding $q$ crosscaps. Compact 2-manifolds are called simply surfaces throughout the paper. 
 The Euler characteristic is defined by
$\chi(\mathbb{S}_h) = 2 - 2h$, $h \geq 0$,  and $\chi(\mathbb{N}_q ) = 2 - q$, $q \geq 1$.
 The Euclidean plane $\mathbb{S}_0$, the projective plane $\mathbb{N}_1$, 
  the torus $\mathbb{S}_1$, and the Klein bottle $\mathbb{N}_2$ are 
  all the surfaces of nonnegative Euler characteristic.

A dominating set for a graph $G$ is a subset $D\subseteq V(G)$ of 
vertices such that every vertex not in $D$ is adjacent to
at least one vertex in $D$. The minimum cardinality of a dominating 
set is called the domination number of $G$ and is denoted by $\gamma (G)$.  
 A variation of domination called Roman domination was introduced by ReVelle \cite{re1,re2}.
 Also see ReVelle and Rosing \cite{rer} for an integer programming
formulation of the problem. 
The concept of Roman domination can be formulated in terms of graphs.
A Roman dominating function (RDF) on a
graph $G$ is a vertex labeling $f : V(G) \rightarrow \{0, 1, 2\}$
 such that every vertex with label $0$ has a neighbor with label $2$. 
 For a RDF $f$, let $ V_i^f = \{v \in V (G) : f(v) = i\}$ for i = 0, 1, 2. 
 Since this partition determines $f$, we can equivalently write 
 $f=(V_0^f; V_1^f; V_2^f)$.
 The weight $f(V(G))$ of a RDF $f$ on $G$ is the value $\Sigma_{v\in V(G)} f (v)$, 
  which equals $|V_1^f| + 2|V_2^f|$. 
The Roman domination number of a graph $G$, denoted by $\gamma_R(G)$, 
is the minimum weight of a Roman dominating function on $G$. 
A function $f=(V_0^f; V_1^f; V_2^f)$  is called a $\gamma_R$-function on $G$, 
if it is a Roman dominating function and $f(V (G)) = \gamma_R(G)$.  
One measure of the stability of the Roman domination number of a graph $G$
  under edge removal is the  Roman bondage number $b_{R}(G)$, defined by  Jafari Rad and Volkmann in \cite{rv0},  
	as  the cardinality of a smallest set of edges whose removal  from $G$ 
   results in a graph with Roman domination number not equal to  $\gamma_R(G)$. 
	For more information 	we refer the reader to 	\cite{akq,bhsx,hx,rv0,rv,samcmj}. 

 In this paper we obtain upper bounds for $b_{R}(G)$  in terms of 
  (a) average degree and maximum degree, 
 and (b) Euler characteristic, girth and maximum degree. 
 We also prove that  the Roman bondage number of every  graph which admits 
  a $2$-cell embedding on a surface with non negative Euler characteristic does not exceed $15$.

\section{Some known results}
\indent 

The following results are important for our investigations.

\begin{them} \label{56}
Let $G$ be a connected graph embeddable on a surface $\mathbb{M}$ whose
 Euler  characteristic $\chi(\mathbb{M})$ is nonnegative  and let $\delta(G) \geq 5$. 
 Then $G$ contains an edge $e=xy$ with $d_G(x) + d_G(y) \leq 11$
 if one of the following holds:
\begin{itemize}
\item[(i)]  \rm{({\bf Wernicke \cite{w} and Sanders \cite{sand}}, respectively)} 
             $\mathbb{M} \in \{\mathbb{S}_0, \mathbb{N}_1\}$.          
\item[(ii)] \rm{({\bf Jendrol' and Voss \cite{jv})} }
            $\mathbb{M} \in \{\mathbb{S}_1, \mathbb{N}_2\}$  and $\Delta (G) \geq 7$. 
\end{itemize}
\end{them}

\begin{lema}[Rad and Volkmann \cite{rv1}]  \label{rv1}
If $G$ is a graph, then $ \gamma_R(G) \leq \gamma_R(G - e) \leq \gamma_R(G) + 1$ for any edge $e \in E(G)$.
\end{lema}

According to the effects of vertex removal on the Roman domination number of a graph $G$, let 

$\bullet$\ $V_{R}^+(G) = \{v \in V(G) \mid \gamma_{R} (G-v) > \gamma_{R} (G)\}$, 

$\bullet$\ $V_{R}^-(G) = \{v \in V(G) \mid \gamma_{R} (G-v) < \gamma_{R} (G)\}$,

$\bullet$\ $V_{R}^0(G) = \{v \in V(G) \mid \gamma_{R} (G-v) = \gamma_{R} (G)\}$.

Clearly $\{ V_{R}^-(G), V_{R}^0(G), V_{R}^+(G)\}$ is a partition of $V(G)$.

\begin{them}[Rad and Volkmann \cite{rv1}]  \label{vvv}
Let $G$ be a graph of order at least $2$. 
\begin{itemize}
\item[(i)] If $v \in V_{R}^+(G)$ then for every  $\gamma_{R}$-function 
$f = (V_0^f; V_1^f; V_2^f)$  on $G$, $|pn_G[v, V_2^f] \cap V_0^f| \geq 3$ 
and $f(v) = 2$. 
\item[(ii)] For any vertex $u \in V(G)$, $\gamma_{R} (G) -1 \leq  \gamma_{R} (G-u) $.
\end{itemize}
\end{them}

\begin{them}[ Hansberg, Rad and Volkmann \cite{hrv}]\label{vc}
Let $v$ be a vertex of a graph $G$. Then
$\gamma_R(G - v)  <  \gamma_R(G)$  if and only if there is a $\gamma_R$-function
$f = (V_0, V_1, V_2)$  on $G$ such that $v \in V_1$.
\end{them}

\begin{them}[Rad and Volkmann \cite{rv0}] \label{erc3} 
If $G$ is a graph, and $x, y,z$ is a path of length $2$ in $G$, then 
\[
b_R(G)  \leq d_G(x) + d_G(y) + d_G(z) - 3 - |N_G(x)\cap N_G(y)|.
\]
\end{them}

The average  degree $ad(G)$ of a graph $G$ is defined as $ad(G) = 2|E(G)|/|V(G)|$. 

\begin{them} [Hartnell and Rall~\cite{hr})] \label{hra}
For any connected nontrivial graph $G$, there exists a pair of vertices, 
say $u$ and $v$, that are either adjacent or at distance $2$ 
from each other, with the property that $d_G(u) + d_G(v) \leq 2ad(G)$.
\end{them}

The girth of a graph $G$ 
is the length of a shortest
cycle in $G$; the girth of a forest is $\infty$. 
\begin{lema}[Samodivkin ~\cite{samajc}]\label{SGZ}
Let $G$ be a connected graph embeddable on a surface $\mathbb{M}$ whose Euler
 characteristic $\chi$ is as large as possible and let the girth of $G$ is $k < \infty$.
 Then:
 \[
  ad(G) \leq \frac{2k}{k-2}(1 - \frac{\chi}{|V(G)|}).
 \]
\end{lema}

Given a graph $G$ of order $n$, let
$\widehat{G}$ be the graph of order $5n$ obtained from $G$ by
attaching the central vertex of a copy of $P_5$, to each vertex of $G$. 

\begin{lema} [S.Akbari, M. Khatirinejad and S. Qajar \cite{akq}]\label{what} 
Let  $G$ be a graph of order $n$, $n \geq 2$. Then $\gamma(\widehat{G}) = 2n$,   
$\gamma_{R} (\widehat{G}) = 4n$ and $b_R(\widehat{G}) = \delta (G) + 2$. 
 \end{lema}

\section{Upper bounds}

A graph $G$ of order at least two is Roman domination vertex critical
 if removing any vertex of $G$ decreases the Roman domination number.
 By $\mathcal{R}_{CV}$ we denote the class of all Roman domination vertex critical graphs. 
 Results on this class can be found in Rad and Volkmann \cite{rv1} and Hansberg et al. \cite{hrv}.

\begin{theorem} \label{erc2}
 Let $G$ be a connected graph. 
\begin{itemize}
\item[(i)]  If $V_{R}^-(G) \not = V(G)$ then 
             $b_{R}(G)              \leq \min\{d_G(u) - \gamma_{R} (G-u) + \gamma_{R} (G) 
             \mid u \in V_{R}^0(G)  \cup V_{R}^+(G)\} \leq \Delta (G)$.
\item[(ii)] If  $b_{R} (G) > \Delta (G)$ then  $G$  is in $\mathcal{R}_{CV}$. 
\end{itemize}
\end{theorem}

\begin{remark} \label{rem1}
Let $G$ be any connected graph of order $n \geq 2$. 
Denote by $S$ the set of all vertices of $\widehat{G}$ each of which is  adjacent to  a vertex of degree $1$. 
Then $f=(V(\widehat{G}) - S; \emptyset; S)$  is a RDF on $\widehat{G}$.
 Since the weight of $f$ is $4n$, by Lemma \ref{what} it follows that 
 $f$ is a $\gamma_{R}$-function on $\widehat{G}$.  
  Theorem \ref{vvv}(i) now implies $V(\widehat{G}) = V^-_{R}(\widehat{G})  \cup V^0_{R}(\widehat{G}) $. 
Since $\gamma_R(P_5) = 4$ and since 
the central vertex of $P_5$ is in $V^0_R(P_5)$, $V(G) \subset V^0_{R}(\widehat{G}) $. 
Labelling the vertices of each $P_5$ of $\widehat{G}$ with $(1, 1, 0, 2, 0)$ yields 
a $\gamma_{R}$-function on $\widehat{G}$. It follows by Theorem \ref{vc} that 
$V(G) = V^0_{R}(\widehat{G}) $. 
All this  together with Lemma \ref{what} shows that the bound  
in Theorem \ref{erc2}(i) is attainable   for all graphs $\widehat{G}$. 
Furthermore, for any graph  $\widehat{G}$ the bound in Theorem \ref{erc3}  
is attainable too. 
\end{remark}

To prove Theorem \ref{erc2}, we need the following lemma:

\begin{lemma} \label{erc21}
Let $G$ be a connected graph. 
For any subset $U \subsetneq V(G)$, let  $E_U$  denote 
the set of all edges each of which joins $U$ and $V(G)-U$. 
\begin{itemize}
\item[(i)] If $v \in V_{R}^0(G) \cup V_{R}^+(G)$ 
         then $\gamma_{R}(G-E_{\{v\}}) > \gamma_{R}(G)$.
\item[(ii)] If $x \in V_{R}^+(G)$ then 
            $1 \leq  \gamma_{R} (G-x) - \gamma_{R}(G) \leq d_G(x) -2$ 
            and for any subset $S \subseteq E_{\{x\}}$ 
            with $|S| \geq d_G(x) - \gamma_{R} (G-v) + \gamma_{R}(G)$,
            $\gamma_{R}(G-S) > \gamma_{R}(G) $.   
\end{itemize}
\end{lemma}
\begin{proof}
(i) We have  $\gamma_{R}(G-E_{\{v\}}) = \gamma_{R}(G-v) + 1 > \gamma_{R}(G)$. 

 (ii)  Denote $p = \gamma_{R} (G-x) - \gamma_{R} (G)$. 
 Let $f$ be any $\gamma_{R}$-function on $G$. 
 Since $p > 0$, by Theorem \ref{vvv}(i) it follows that  $f(x) = 2$. 
Hence   $h = (V_0^f - N_G(x); V_1^f \cup (N_G(x) - V_2^f); V_2^f-\{x\})$ is a RDF  on $G-x$. 
 But then $\gamma_{R}(G) + p = \gamma_{R} (G-x) \leq h(V(G-x))  \leq \gamma_{R}(G) + d_G(x) -2$. 
 Hence $1 \leq p \leq d_G(x) -2$. 
For any set $S \subseteq E_{\{x\}}$ with $|S| \geq d_G(x) - p$ we have $\gamma_{R}(G-S) \geq \gamma_{R}(G-E_{\{x\}}) - |E_{\{x\}}| + |S| 
 \geq (\gamma_{R}(G-x) + 1) - d_G(x) + (d_G(x) -p) = 
 \gamma_{R}(G) +1$, where the first inequality follows from Lemma \ref{rv1}.   
 \end{proof}

\begin{proof}[Proof of Theorem \ref{erc2}]
(i) The result follows immediately by Lemma \ref{erc21}. 

(ii) Immediately by (i).   
 \end{proof}
 
 Rad and Volkmann \cite{rv}  as well as Akbari et al. \cite{akq}
 gave upper bounds on the Roman bondage number of planar graphs. 
 Upper bounds on the Roman bondage number of graphs 2-cell embeddable on
topological surfaces in terms  of orientable/non orientable genus and maximum degree,  
  are obtained by the present author in \cite{samcmj}.

\begin{theorem} \label{restrtot}
Let  $G$ be a connected graph with $\Delta (G) \geq 2$. 
\begin{itemize}
\item[(i)] Then $b_{R}(G)  \leq 2ad(G) + \Delta(G) -3$.
\item[(ii)] Let $G$ be embeddable on a surface $\mathbb{M}$ whose Euler
 characteristic $\chi$ is as large as possible. 
 If $G$ has order $n$ and girth $k < \infty$ then:
 \[
  b_{R}(G) \leq \frac{4k}{k-2}(1 - \frac{\chi}{n}) + \Delta(G) -3 \leq -\frac{12\chi}{n} + \Delta(G) +9.
 \]
\end{itemize}
\end{theorem}
\begin{proof}
(i)  If $G$ is a complete graph then  the result is obvious. 
Hence we may assume $G$ has nonadjacent vertices. 
 Theorem \ref{hra} implies that there are 
$2$ vertices, say $x$ and $y$, that are either adjacent
or at distance 2 from each other, with the property that 
$d_G(x) + d_G (y) \leq 2ad (G)$. 
 Since $G$ is connected and $\Delta (G) \geq 2$, 
 there is a vertex $z$ such that $xyz$ or $xzy$ is a path.  
 In either case by Theorem \ref{erc3}   we have 
 $b_{R}(G) \leq d_G(x) + d_G(y) + d_G(z) - 3 \leq 2ad(G) + \Delta(G) -3$. 
 
 (ii) Lemma \ref{SGZ} and (i) together imply the result. 
\end{proof}

\begin{remark} \label{dmax}
Let $\mathbb{M}$ be a surface.   
Denote $\delta_{max}^{\mathbb{M}} = \max\{\delta (H) \mid \mbox{a graph \ } H \mbox{\ is \ } 
2\mbox{-cell\ }$ $\mbox{embedded in\ }  \mathbb{M}\}$.
Let $G$ be a connected graph $2$-cell embeddable on $\mathbb{M}$
 and $\delta (G) = \delta_{max}^{\mathbb{M}}$. 
 By Lemma \ref{what} it immediately follows 
 $b_R(\widehat{G}) = \delta_{max}^{\mathbb{M}} +2$. 
 Note that (a) if $\chi(\mathbb{M}) \leq 1$ then  
$\delta_{max}^{\mathbb{M}} \leq \left\lfloor (5+\sqrt{49-24\chi(\mathbb{M})})/2\right\rfloor$
 (see Sachs \cite{Sachs}, pp. 226-227), 
 and (b) it is well known  that $\delta_{max}^{\mathbb{S}_0} = \delta_{max}^{\mathbb{N}_1} = 5$ and
  $\delta_{max}^{\mathbb{N}_2} = \delta_{max}^{\mathbb{N}_3} = \delta_{max}^{\mathbb{S}_1} =6$.
\end{remark}

In \cite{akq}, Akbari, Khatirinejad and Qajar recently prove that $b_R(G) \leq 15$ provided $G$ is a planar graph. 
As the next result shows, more is true.

\begin{theorem} \label{ppp}
Let $G$ be a connected graph $2$-cell embedded on a surface $\mathbb{M}$ with non negative 
Euler characteristic  and let $\Delta (G) \geq 2$.  Then $b_R(G) \leq 15$. 
	 \end{theorem}

\begin{proof}
If $2 \leq \Delta (G) \leq 6$ then  $b_{R}(G)   \leq 3\Delta(G) -3 \leq 15$, because of Theorem \ref{erc3}. 
So, assume $\Delta (G) \geq 7$. 
Denote $V_{\leq 5}= \{v \in V(G) \mid d_G(v) \leq 5\}$   and $G_{\geq 6} = G - V_{\leq 5}$. 
 Since $\chi(\mathbb{M}) \geq 0$, $\delta(G) \leq 6$ (see Remark \ref{dmax}).  
 If $\delta (G)=6$ then $G$ is a $6$-regular triangulation on the torus or in the Klein Bottle, 
 a contradiction with $\Delta (G) \geq 7$. 
 So, $\delta (G) \leq 5$ and then $V_{\leq 5}$ is not empty. 
 Since $G_{\geq 6}$ is embedded without crossings on $\mathbb{M}$ and $\chi (\mathbb{M}) \geq 0$, 
 there is a vertex $u \in V(G_{\geq 6})$ with $d_{G_{\geq 6}}(u) \leq 6$. 
  If $u$ has exactly $2$ neighbors belonging to $V_{\leq 5}$ then 
 again by Theorem \ref{erc3}, $b_{R}(G) \leq 15$.
 Now let all $v_1, v_2, v_3 \in V_{\leq 5}$ be adjacent to $u$.
 Denote by $E_1$ the set of all edges of $G$ which are incident to at least one of $v_1, v_2$ and $v_3$.
 Since $v_1, v_2$ and $v_3$ are isolated vertices in $G-E_1$, 
 for any $\gamma_{\mathcal{P}R}$-function $g$ on $G-E_1$, $g(v_1) = g(v_2) = g(v_3)=1$. 
 Define now $f : V(G) \rightarrow \{0, 1, 2\}$ by $f(v_1) = f(v_2) = f(v_3) = 0$, $f(u) = 2$ 
 and $f(v) = g (v)$ for every $v \in V(G) - \{u, v_1, v_2, v_3\}$. 
 Clearly $f$ is a RDF on $G$ with $\gamma_{R}(G) \leq f(V(G)) < g(V(G-E_1)) = \gamma_{R}(G-E_1)$.
 Thus, $b_R(G) \leq |E_1| \leq d_G(v_1) + d_G(v_2) + d_G(v_3) \leq 15$. 

 So, it remains to consider the case where each vertex of degree at most $6$ in 
$G_{\geq 6}$  has   no more than one neighbor in $V_{\leq 5}$. 
 It immediately follows that $\delta(G_{\geq 6}) \geq 5$. 
     First assume $\delta(G_{\geq 6})=5$. 
     By Theorem\ref{56} it follows that there is an edge $xy \in E(G_{\geq 6})$ such that 
     $d_{G_{\geq 6}}(x)  + d_{G_{\geq 6}} (y) \leq 11$. 
     Hence $d_G(x) + d_G (y) \leq 13$. Let without loss of generality 
     $d_{G_{\geq 6}}(x)  \leq d_{G_{\geq 6}} (y)$. 
     Then $x$ has exactly one neighbor in $V_{\leq 5}$, say $v$.
          By Theorem \ref{erc3} applied to the path $v, x, y$ we have $b_{R} (G) \leq 5 + 13 - 3 = 15$.
Now let $\delta(G_{\geq 6}) \geq 6$. 
     But then $G_{\geq 6}$ is a $6$-regular triangulation on the torus or in the  Klein bottle. 
		Since $\Delta (G) \geq 7$, $G \not= G_{\geq 6}$ and there is 
		 a path $x,y,z$ in $G$, where $d_G(z) \leq 5$,  and both $x$ and $y$ are in $V(G_{\geq 6})$. 
		Since clearly $|N(x) \cap N(y)| \geq 2$, again using Theorem \ref{erc3} 
		we obtain $b_{\mathcal{P}R} (G) \leq  7 + 7 + 5 - 3 -2 = 14$. 
\end{proof}

\end{document}